\documentclass{amsart}

\usepackage{enumitem}
\usepackage[charter]{mathdesign}
\usepackage{amsmath}

\usepackage{amsfonts}
\usepackage{graphicx, subfigure}
\usepackage{graphics}

\usepackage[usenames, dvipsnames]{color}
\usepackage[colorlinks=true,
        raiselinks=true,
        linkcolor=MidnightBlue,
        citecolor=ForestGreen,
        urlcolor=Mahogany,
        plainpages=false]{hyperref}

\usepackage{tikz}
\usepackage[all]{xy}

\usepackage{tikz}
\usepackage{epstopdf}
\usetikzlibrary{arrows}
\usepackage{psfrag}

\numberwithin{equation}{section}

{ \theoremstyle{plain}
  \newtheorem{theorem}{Theorem}
  \newtheorem{lemma}{Lemma}
  \newtheorem{corollary}{Corollary}
  
  \newtheorem{proposition}{Proposition}
  \newtheorem{assumption}{Assumption}
  
  \newtheorem{remark}{Remark}
  \newtheorem{definition}{Definition}
}

\newcommand{\dt}{{\tt dt}}
\newcommand{\Next}{\mathsf{X}}
\newcommand{\U}{\;\mathsf{U}\;}
\newcommand{\Un}{\;\mathsf{U}^{\leq n}\;}

\newcommand{\AP}{\texttt{AP}}
\newcommand{ \pr}{\mathsf{P}}

\def\e{\mathrm{e}}

\def\B{\mathbb{B}}

\def\P{\mathbb{P}}

\def\R{\mathbb{R}}
\def\N{\mathbb{N}}

\def\d{\mathrm{d}}

\def\f{\mathcal{F}}
\def\e{\mathcal{E}}
\def\i{\mathcal{I}}
\def\j{\mathcal{J}}

\def\x{\mathbf{x}}

\def\b{\mathcal{B}}

\def\ve{\varepsilon}
\newcommand{\las}{\mathrm{l.a.s.}}

\title[Infinite-horizon specifications over Markov processes]{Characterization and computation of infinite-horizon specifications over Markov processes}

\begin{document}

\author{Ilya Tkachev and Alessandro Abate}
\thanks{
\hspace{-0.6cm}
I. Tkachev is with the Delft Center for Systems \& Control, Delft University of Technology, The Netherlands. Email: \texttt{i.tkachev@tudelft.nl}. \\
A. Abate is with the Department of Computer Sciences, University of Oxford, United Kingdom,
and with the Delft Center for Systems \& Control, Delft University of Technology, The Netherlands. Email: \texttt{alessandro.abate@cs.ox.ac.uk}.
}


\begin{abstract}
  This work is devoted to the formal verification of specifications over general discrete-time Markov processes, with an emphasis on infinite-horizon properties.
  These properties, formulated in a modal logic known as PCTL, can be expressed through value functions defined over the state space of the process.
  The main goal is to understand how structural features of the model (primarily the presence of absorbing sets) influence the uniqueness of the solutions of corresponding Bellman equations. Furthermore, this contribution shows that the investigation of these structural features leads to new computational techniques to calculate the specifications of interest: the emphasis is to derive approximation techniques with associated explicit convergence rates and formal error bounds.  \\[1ex]
\smallskip
\noindent \textbf{Keywords:}
discrete-time Markov processes,
PCTL model checking,
infinite-horizon properties,
Bellman equations,
absorbing sets.
\end{abstract}

\maketitle

\section{Introduction}

The use of formal verification notions and methods for dynamical systems has recently become an
active inter-disciplinary area of research in systems and control theory \cite{t2009}.
One of the most efficient techniques is model-checking,
which
aims at determining the satisfaction set of a given specification,
i.e. the set of all states that initialize realizations verifying that specification.
Probabilistic Computation Tree Logic (PCTL) is a modal logic which is widely used in formal verification
to express specifications for discrete-time probabilistic processes \cite[Chapter 10]{bk2008}.
The special case of discrete-time Markov Chains (\dt-MC) -- models over discrete (countable) spaces -- is well-studied in the literature and PCTL specifications can be verified over these models in an automatic manner by employing computationally advantageous probabilistic model checking techniques \cite{HKNP06,kkz2005}.
PCTL model checking has also been validated over numerous compelling applications \cite{fknp2011}.

The formal extension of PCTL to discrete-time Markov processes (\dt-MP) over general (uncountable) state spaces has recently been discussed in \cite{h2005, rcsl2010}.
The latter work in particular has expressed the satisfaction set of a given PCTL specification as the level set of an associated state-dependent value function,
and has further characterized the computation of such value function via dynamic programming (DP) \cite{bs1978}. Within PCTL, there is a clear distinction between finite-horizon specifications (the satisfiability of which depends on finite realizations of the system) and infinite-horizon specifications (those characterized over infinite paths). In the context of \dt-MC with a finite state space, DP over a finite horizon is performed by iterative matrix multiplications, whereas DP over an infinite horizon is reduced to solving systems of linear equations.
On the other hand, over a general state space the corresponding procedures
-- namely Bellman iterations and Bellman equations -- involve integral operators.
Recent work (see e.g. \cite{APLS08b}) has shown that explicit analytical solutions over uncountable state-spaces are not to be found in general,
and has stressed the need for methods to compute value functions with any given precision.

In the context of \dt-MP,
the work in \cite{h2005} has put forward finite abstractions,
where measures are approximated by monotone functions of sets.
Although these abstractions are sound and upper and lower bounds for the expression of value functions have been derived \cite[Theorem 33]{h2005},
no method to tune them has been provided.
Also, their tightness, usefulness, or possible triviality (i.e. conditions for the error bounds to be less than $1$) have not been addressed.
The work in \cite{rcsl2010}, in turn, has characterized PCTL specifications and their associated value functions with an emphasis on the issue of uniqueness of solutions of the related Bellman equations.
The following questions have been left open to investigation:
\begin{enumerate}
  \item how to compute finite-horizon value functions in PCTL with any given precision;
  \item since in general value functions are not known exactly and satisfaction sets are expressed as level sets of these functions, how to verify nested PCTL formulae (namely, specifications where the satisfaction set for the first formula appears in the definition of a second one);
  \item how to verify infinite-horizon PCTL specifications, particularly if the sufficient conditions for the uniqueness of solutions of Bellman equations in \cite{rcsl2010} are not satisfied.
\end{enumerate}

With focus on 1), finite-horizon computations have recently received considerable attention.
For discrete-time Stochastic Hybrid Systems (a class of \dt-MP),
the work in \cite{aklp2010} has put forward finite abstraction techniques to perform DP iterations over corresponding finite state-space \dt-MC.
These results have been further sharpened in \cite{sa2011},
where abstractions by state-space partitioning are obtained adaptively, in accordance to a specification-dependent error.
In both works the explicit abstraction error grows linearly with the time horizon of the corresponding PCTL specification, which does not allow applying the developed methods directly to the verification of infinite-horizon properties.

The contribution of this work is hence focused on questions 2) and 3) and is twofold: the first goal is to complete the formal discussion on general state-space PCTL verification by dealing with nested formulae; the second goal (and the main task of this work) is to provide both analysis and computational tools for infinite-horizon PCTL specifications under conditions on the model that are as weak as possible and that are easy to verify.

In order to address question 2), we introduce the concepts of sub- and super-satisfaction sets for PCTL specifications,
the characterization of which requires only approximate knowledge of the corresponding value functions.
Specifically, we show how the sub- and super-satisfaction sets of a nested sub-formula propagate to the corresponding sets of the main formula:
this is achieved by using monotonicity properties of corresponding value functions.

In order to tackle question 3), we extend and generalize recent results in \cite{ta2011,ta2012}, showing that the sufficient condition provided in \cite{rcsl2010} for the uniqueness of the solution of a Bellman equation is only satisfied if the solution is trivial in some sense. We further show that a weaker version of this condition is both necessary and sufficient if the \dt-MP admits certain continuity properties.
This result leads to novel techniques to solve Bellman equations whenever their solution is not unique,
and provides approximation techniques with associated explicit convergence rates and error bounds.
These techniques are based on the reduction of the infinite-horizon problem to a finite-horizon one,
for which computational methods available in the literature \cite{aklp2010,sa2011} can be directly applied.
We furthermore discuss the relationship between the issue of uniqueness of solution and the presence of absorbing sets over the (uncountable) state space:
absorbing sets are shown to play a fundamental role for both the characterization and the computation of infinite-horizon PCTL properties.

The contribution is organized as follows. Section \ref{sec:pctl} introduces discrete-time Markov processes and PCTL specifications, and discusses the verification of nested PCTL formulae. Section \ref{sec:ver} dives in depth into infinite-horizon problems. Section \ref{sec:cs} provides two case studies to discuss the results and finally Section \ref{sec:concl} concludes the work.

Throughout the article we use tools of measure theory and of functional analysis. The following references can be consulted: \cite{D04} for probability theory,
\cite{r1984} for Markov processes and \cite{r1987} for functional analysis and measure theory.

\section{Markov processes and PCTL}\label{sec:pctl}

\subsection{Discrete-time Markov processes}

Let $(X,\b)$ be some measurable space and let $P:X\times \b \to [0,1]$ be a stochastic kernel, so that $P(\cdot,B)$ is a non-negative measurable function for any set $B\in \b$ and $P(x,\cdot)$ is a probability measure on $(X,\b)$ for any $x\in X$. The space of trajectories is denoted by $\Omega := X^{\N_0}$ (here $\N_0 := \N \cup\{0\}$) and its product  $\sigma$-algebra by $\f$.
It follows from \cite[Theorem 2.8]{r1984} that there exists a unique discrete-time Markov process (\dt-MP) $\x = (\x_n)_{n\geq 0}$ with the transition kernel $P$,
that is,
for any $x\in X$ there exists a unique probability measure $ \pr_x$ on $(\Omega,\f)$ such that $ \pr_x(\x_0 = x) = 1$, and for any measurable set $B\in \b$ and any time epoch $n\geq 0$
\begin{equation}\label{eq:Markov}
     \pr_x(\x_{n+1}\in B|\x_0,\x_1,\dots,\x_n) =  \pr(\x_{n+1}\in B|\x_n) = P(\x_n,B).
\end{equation}
Equation \eqref{eq:Markov} characterizes the Markov property and it indicates that the future of the process $\x_{n+1}$ is independent of its past history $(\x_0,\dots,\x_{n-1})$, given its current value $\x_n$. As a result, any \dt-MP can be characterized equivalently by the triple $(X,\b,P)$.

A familiar class of \dt-MP is that of stochastic dynamical systems.
If $(\xi_n)_{n\geq 0}$ is a sequence of iid random variables and $f:X\times \R\to X$ is a measurable map, then
\begin{equation}\label{eq:dynamical}
  \x_{n+1} = f(\x_n,\xi_n),\qquad \x_0 = x \in X,
\end{equation}
is always a Markov process characterized by a kernel
\begin{equation*}
  Q(x,A) := \nu(\{\xi\in \R:f(x,\xi)\in A\})
\end{equation*}
where $\nu$ is the distribution of $\xi_0$.
Conversely, under some mild conditions on the structure of the state space,
any \dt-MP $X$ admits a dynamical representation as in \eqref{eq:dynamical},
for an appropriate choice of the function $f$ \cite[Proposition 8.6]{k1997a}.
However, theoretical studies of \dt-MP,
as well as the current article,
usually employ the representation via stochastic kernels.

The reader interested in further discussions about modeling aspects of \dt-MP is referred to \cite[Appendix A1]{m2008}.
Among other models related to \dt-MP, Labeled Markov Processes (LMP) \cite{dgjp04} are of interest as they embed non-determinism and allow for sub-stochastic transition kernels.

\subsection{Probabilistic Computation Tree Logic (PCTL)}

PCTL is a modal logic employed to characterize classes of temporal properties of \dt-MC \cite{bk2008} and of \dt-MP \cite{h2005,rcsl2010}.
Properties are expressed as formulae in PCTL and are constructed according to the grammar of this logic.
The grammar is based on $\AP$, the set of \textit{atomic propositions}, which can be thought of as tags or labels associated to the states.
Let $A\in \AP$ and $x\in X$; we write $x\models A$ if the atomic proposition $A$ is valid at state $x$.
Since there is no substantial difference between $A$ and its \textit{satisfaction set} $\{x\in X:x\models A\}\subseteq X$, we define atomic propositions to be measurable subsets of $X$, or equivalently $\AP\subseteq \b$, and require that  $X\in \AP$.
The grammar of PCTL is defined as follows.
Atomic propositions are basic formulae that are used to build more complex formulae via logical rules.
PCTL \textit{state} \textit{formulae} are subsets of $X$, whereas \textit{path formulae} are subsets of $\Omega$. More precisely:
\begin{itemize}
  \item each atomic proposition $A\in \AP$ is a formula with $A$ itself as its satisfaction set;
  \item if $A$ and $B$ are formulae, then so are $\neg A$ and $A\wedge B$;
  \item if $ P hi$ is a path formula and $p\in [0,1]$, then $\P_{\bowtie p}[ P hi]$ is a (state) formula, where $\bowtie$ can be any symbol from the collection
   $\{\leq,<,\geq,>\}$;
  \item if $A$ and $B$ are formulae and $n\in \N_0$, then $\Next\, A$, $A \Un B$, and $A \U B$ are path formulae.
\end{itemize}
The semantics of PCTL state formulae is given as follows:
\begin{align*}
    x \models X  &\quad {\rm for\ all} \quad  x \in X\\
    x \models A                                & \quad \Leftrightarrow \quad  x \in A \\
    x \models \lnot A                  & \quad \Leftrightarrow \quad  x \in A^c := X\setminus A\\
    x \models A \land B        & \quad \Leftrightarrow \quad  x \in A \cap B\\
    x \models \P_{\bowtie p}[ P hi]            & \quad \Leftrightarrow \quad   \pr_x(  P hi ) \bowtie p
\end{align*}
With regards to path formulae,
the meaning of $\Next\, A$ (the \textit{next} operator) is $\x_1 \in A$,
thus $x\models \P_{\bowtie p}[\Next A]$ if and only if $P(x,A)\bowtie p$.
The two additional path formulae depend on the \textit{bounded until} operator $\Un$ and on the \textit{unbounded until} operator $\mathsf U$.
In order to characterize them through subsets of $\Omega$, let us introduce for any set $A\in \b$
\begin{equation*}
    \tau_A := \inf\{n\geq 0: \x_n\in A\}
\end{equation*}
to be the first hitting time of a set $A$ over a realization $\x_0,\x_1,\dots$. Clearly, $\tau_A$ is a random variable with values in $\N_0\cup\{\infty\}$. We define $A \Un B := \{\tau_{B} \leq \tau_{A^c},\tau_B\leq n \}$,
so that $A \Un B \in \f$ whenever $A,B\in \b$,
which means that the path formula is satisfied over a trajectory for which $B$ holds at least once within the $n$-step horizon, while $A$ is persistently valid until that moment.
Similarly, for the infinite-horizon case, we define
\begin{equation*}
  A\U B := \{\tau_{B}\leq\tau_{A^c},\tau_B<\infty\}.
\end{equation*}

To characterize satisfaction sets for until operators, we introduce the so called \textit{reach-avoid}\footnote{Alternatively known as constrained reachability \cite{bk2008}.} value functions:
for any $A,B\in \b$, let us define
\begin{equation*}
    w_n(x;A,B) :=  \pr_x\left(A \Un B\right),\quad w(x;A,B) :=  \pr_x\left(A \U B\right),
\end{equation*}
which leads to expressing $\P_{\bowtie p}[A \Un B] = \left\{x\in X:w_n(x;A,B)\bowtie p\right\}$.
Functions $w_n,w$ are measurable, thus all PCTL formulae are well-defined measurable subsets of $X$ and all path formulae are elements of $\f$ \cite{rcsl2010}.\footnote{Although the theory in \cite{rcsl2010} has been developed for models with $X$ carrying a topological structure, all the results on measurability hold without this requirement and as such they are also valid in the present instance.
This work resorts to a topological structure over the state space only in Section \ref{ssec:abs}.}

Let us provide a few examples: if $A,B$ are PCTL formulae, then $\P_{\leq0.05}\left[A \U \P_{<1}[\Next\, B]\right]$ is a PCTL formula. Likewise, $A\Rightarrow \P_{\geq 0.95}[A\U B]$ is a PCTL formula, since $A\Rightarrow B = \neg A\vee B$ and $A\vee B = \neg(\neg A\wedge\neg B)$. However, $\P_{>0}\left[(\Next\, A)\wedge (B\U C)\right]$ is not a PCTL formula, since the logical operation $\wedge$ is defined for state formulae but not over path formulae. Furthermore, PCTL path formula $\lozenge^{\leq n} A:=X \Un A = \{\tau_A\leq n\}$ is known as a \textit{reachability} event for a given set $A$ and relates to a wide and important class of problems in systems and control \cite{APLS08b}. Its dual, the \textit{invariance} (or \textit{safety}) event $\square^{\leq n} A := \neg \left(\lozenge^{\leq n} A^c \right)=\{\tau_{A^c} >n\}$, cannot be directly expressed in PCTL since the negation of path formulae is not allowed. On the other hand,
\begin{equation*}
     \pr_x(\tau_{A^c}>n) = 1-w_n(x;X,A^c),
\end{equation*}
thus one can define $\P_{\bowtie p}\left[\square^{\leq n} A\right] = \P_{\bowtie'1-p}\left[X\Un A^c\right]$,
where the symbol $<'$ stands for $\geq$, the symbol $\leq'$ stands for $>$ etc.
We denote the invariance value functions by
\begin{equation}\label{eq:u-w}
u_n(x;A) := 1-w_n(x;X,A^c),\quad u(x;A) := 1-w(x;X,A^c).
\end{equation}
The results for reach-avoid and invariance given in this work can thus be directly exported to the reachability property.
The latter represents also a crucial property for other types of logics, for instance linear temporal logic (LTL) \cite[Chapter 5]{bk2008}.
In particular, \cite{akm2011} has argued that the verification of a subclass of LTL specifications over a \dt-MP can be reduced to a reachability problem \cite[Theorem 4]{akm2011}.

\subsection{Nested PCTL properties}\label{ssec:nest}

As mentioned in the introduction,
it is in general not expected that the value functions $w_n$ and $w$ can be expressed explicitly.
An alternative goal is the following \cite{aklp2010}:
given any precision level $\delta>0$, find approximate functions $\hat w_n$ and $\hat w$
such that $|\hat w_n(x)-w_n(x)|\leq \delta$ and $|\hat w(x) - w(x)|\leq \delta$,
for all $x\in X$.
Consider however the formula $\P_{\geq p_1}\left[A\U \P_{\leq p_2}[B\U C]\right]$:
if the value function $w(x;B,C)$ can only be characterized approximately,
what set should be considered to characterize $\P_{\leq p_2}[B\U C]$?
And how could this set be used in the parent formula?
To resolve this issue we need the following fact.
\begin{proposition}\label{prop:subsat}
    Let $A\subseteq A^*$ and $B\subseteq B^*$ be elements of $\b$ and let $n\in \N_0$. For all $x\in X$:
    \begin{equation*}
      w_n(x;A,B)\leq w_n(x;A^*,B^*),\quad w(x;A,B)\leq w(x;A^*,B^*).
    \end{equation*}
\end{proposition}

\begin{proof}
Since $\{\tau_{B} \leq \tau_{A^c},\tau_B\leq n \}\subseteq \{\tau_{B^*} \leq \tau_{(A^*)^c},\tau_{B^*}\leq n \}$ the proof immediately follows as the probability measure $ \pr_x$ is a monotonic function of sets for any $x\in X$.
\end{proof}

For a PCTL formula $A\in \b$, we say that $A_*$ ($A^*$) is a \textit{subsatisfaction} (\textit{supersatisfaction}) set if $A_*\subseteq A$ ($A\subseteq A^*$). Clearly, $A_*$ denotes a conservative set, the states of which also satisfy $A$, while $A^*$ denotes a relaxed set: any state in $(A^*)^c$ does not satisfy $A$ either\footnote{The approach here is similar to three-valued approximations in \cite{flw2006}, \cite[Section 4]{h2005}, and \cite{hpw2009}.}.

As done above, let $\hat w_n,\hat w$ denote some abstract $\delta$-approximations of $w_n$ and $w$, respectively.
Let us show as an example,
how the formula $\P_{\geq p_1}\left[A\U \P_{\leq p_2}[B\U C]\right]$ can be verified.
Since for all $x\in X$ it holds that
\begin{equation*}
    \hat w(x;B,C)-\delta\leq w(x;B,C)\leq \hat w(x,B,C)+\delta,
\end{equation*}
it follows that $\hat w(x;B,C)\leq p_2-\delta$ implies $w(x;B,C)\leq p_2$, and that $\hat w(x;B,C)>p_2+\delta$ implies $w(x;B,C)>p_2$.
As a result, if we denote $D = \P_{\leq p_2}[B\U C]$,
then the sets
\begin{equation*}
    D_* := \{x\in X: \hat w(x;B,C)\leq p_2-\delta\},\quad D^* := \{x\in X: \hat w(x;B,C) \leq p_2+\delta\}
\end{equation*}
represent sub- and super-satisfaction sets for $D$. Finally, from Proposition \ref{prop:subsat} we obtain:
\begin{equation*}
    E_* := \{x\in X: \hat w(x;A,D_*)\geq p_1+\delta\}, \quad E^{*} := \{x\in X: \hat w(x;A,D^*) \geq p_1-\delta\},
\end{equation*}
which are sub- and super-satisfaction sets for $\P_{\geq p_1}\left[A\U \P_{\leq p_2}[B\U C]\right]$.
The application of this procedure over formulae including the operator $\Next$ is direct, since $P(x,\cdot)$ is a monotonic function of a set-valued argument for any $x\in X$.

A general algorithm for the verification of nested formulae follows: given the ability to approximately compute value functions with a precision $\delta$, find sub- and super-satisfaction sets for the sub-formulas on the lowest level (leaves) of a given formula tree, then use these sets to find sub- and super-satisfaction sets for higher-level formulae inductively, until the sub- and super-satisfaction sets for the given formula are found (at the root).
Note the similarity between our approach and the three-valued approximations of PCTL.

\section{Verification of infinite-horizon PCTL specifications}
\label{sec:ver}
The goal of this section is to investigate the verification of infinite-horizon PCTL specifications and to provide methods to compute associated value functions with any given precision. For this purpose Section \ref{ssec:dp} introduces DP techniques to characterize the corresponding value functions via Bellman recursions and fixpoint equations,
points out related issues in their evaluation and provides sufficient conditions for the precise reduction of infinite-horizon problems to finite-horizon ones. In Section \ref{ssec:abs}, the concept of absorbing set is used to show that for a class of problems the aforementioned conditions are also necessary, and that they relate to the uniqueness of the solution of Bellman fixpoint equations. This result is further applied to derive methods to solve Bellman equations with non-unique solutions, both in the general case (which is done leveraging Lyapunov-like locally excessive functions -- cfr. Section \ref{ssec:decomp}), and in the special case of integral kernels (where such functions are not needed -- cfr. Section \ref{ssec:integral}).
The presented techniques depend on the characterization of absorbing sets,
which is discussed in Section \ref{ssec:v.o.s.}.
The obtained results are further compared with approaches in the literature on \dt-MP in Section \ref{ssec:connections}.

\subsection{Dynamic programming and Bellman equations}\label{ssec:dp}

Let $\B$ denote the space of all real-valued, bounded and measurable functions on $X$. It is a Banach space with a norm given by $\|f\| := \sup_{x\in X}|f(x)|$ for $f\in \B$.
An operator $\j:\B\to\B$ is called linear if
\begin{equation*}
  \j(\alpha f+\beta g) = \alpha \j(f)+\beta \j(g)
\end{equation*}
for any constants $\alpha$, $\beta\in \R$ and functions  $f$, $g\in\B$. The quantity
\begin{equation}\label{eq:operator-norm}
    \|\j\| := \sup\limits_{\|f\|\leq 1}\|\j f\|
\end{equation}
is called the norm of the linear operator $\j$.
We say that a linear operator $\j$ is a contraction whenever it holds that $\|\j\|<1$.
An important example of a linear operator associated to a \dt-MP is the transition operator $ P :\B\to\B$,
which is induced by the kernel $P$.
The action of operator $ P $ on function $f\in \B$ is given by the following formula:
\begin{equation*}
   P  f(x) := \int_X f(y)P(x,\mathrm dy).
\end{equation*}
Let us furthermore introduce an invariance operator $\i_A$,
parameterized by a measurable set $A\in \b$,
and given by $\i_A f(x) = 1_A(x) P  f(x)$. Clearly, $\i_A$ is also a linear operator and $\i_X =  P $. Moreover, $\i_A$ is a monotone operator, which means that for all functions $f,g\in \B$ and any set $A\in \b$ it holds that $\i_A f(x)\leq \i_A g(x)$ for all $x\in X$ whenever it holds that $f(x)\leq g(x)$ for all $x\in X$.

As an abbreviation, for a function $g:X\to\R$ and a constant $\delta\in \R$ we further write $\{g\leq \delta\} := \{x\in X:g(x)\leq \delta\}$; a similar notation is used for any of the other symbols in the collection $\{<,\geq,>,=\}$.

Let us introduce a DP procedure for until-like specifications in PCTL.
Let $A,B\in \b$ be given sets (equivalently, state formulae in PCTL).
From \cite{rcsl2010,SL10} it follows:
\begin{equation}\label{eq:DP-w}
\begin{cases}
  w_{n+1}(x;A,B) &= 1_B(x)+\i_{A\setminus B} w_n(x;A,B),
  \\
  w_0(x;A,B) &= 1_B(x).
\end{cases}
\end{equation}
The computation in \eqref{eq:DP-w} involves iterations of the integral operator $\i_{A\setminus B}$.
Results in \cite{aklp2010,sa2011} allow one to compute a piece-wise constant function approximation $\hat w_n$,
which is such that $\|\hat w_n - w_n\|\leq \lambda n$,
where the constant $\lambda$ depends on the quality of the state space partitioning (see e.g. \cite[Theorem 4]{sa2011}).
Thus, in the remainder of this work we assume that finite-horizon problems can be solved approximately and with any given precision by any of the techniques given in the literature,
and instead focus on the reduction of infinite-horizon problems to finite-horizon ones.

For infinite-horizon problems, it holds that $w(x;A,B) = \lim_{n\to\infty} w_n(x;A,B)$, where the limit is point-wise non-decreasing \cite{rcsl2010}.
In \cite[Lemma 5]{rcsl2010} the monotone convergence theorem is applied to $w_n\to w$,
in order to show that the function $w$ solves the fixpoint Bellman equation
\begin{equation}\label{eq:Bellman-w}
w(x;A,B) = 1_B(x)+\i_{A\setminus B} w(x;A,B).
\end{equation}
However the convergence of $w_n\to w$ is not necessarily uniform.
Moreover, equation \eqref{eq:Bellman-w} may have multiple solutions:
since it is an affine equation,
if it does not have a unique solution then it admits infinitely many,
spanning an affine subspace of $\B$.
To further look into this issue we leverage value functions for invariance.
As discussed above,
the until specification can be used to express the invariance over a given set $A\in \b$.
Using formulae \eqref{eq:u-w} and \eqref{eq:DP-w} we obtain the following DP recursion
\begin{equation}\label{eq:DP-u}
\begin{cases}
  u_{n+1}(x;A) &= \i_{A} u_n(x;A),
  \\
  u_0(x;A) &= 1_A(x).
\end{cases}
\end{equation}
It easily follows that $u_n$ converges point-wise non-increasingly to function $u$, thus
\begin{equation}\label{eq:Bellman-u}
u(x;A) = \i_A u(x;A).
\end{equation}
Clearly, the verification of the invariance specification inherits issues of non-uniform convergence and of non-uniqueness of the Bellman equation \eqref{eq:Bellman-u} from the until specification in \eqref{eq:Bellman-w}.
However, the Bellman equation for the invariance specification has the advantage of being linear and thus always admits the trivial solution $u\equiv 0$.
Moreover, the analysis of the affine equation on a linear space can be reduced to the analysis of its homogeneous (linear) version:
dealing with \eqref{eq:Bellman-u} leads to finding methods for solving \eqref{eq:Bellman-w} as well.

\begin{remark}\label{rem:least-fixpoint}
There exists a least fixed-point characterization for the infinite-horizon value functions \cite[Lemma 6]{rcsl2010}: $w(x;A,B)$ is the least non-negative solution of \eqref{eq:Bellman-w}, i.e. if $f$ is any other non-negative solution of \eqref{eq:Bellman-w},
then $w(x;A,B)\leq f(x)$ for all $x\in X$. As a result, $u(x;A)$ is the
largest solution of \eqref{eq:Bellman-u} not exceeding $1$.
Although such characterization adds little to the computation of $u$ and $w$,
it results in the useful fact that $\|u\| = 1$ whenever $u$ is non-trivial, namely whenever $u$ is not identically equal to zero.
\end{remark}

One sufficient condition for the uniqueness of the solution of \eqref{eq:Bellman-u} is given as follows:
$\|u_1(\cdot,A)\|<1$ \cite[Proposition 7]{rcsl2010},
which in turn leads to the contractivity of the operator $\i_A$.
While this condition may be easy to check, it can be restrictive:
in this case \eqref{eq:Bellman-u} admits the unique solution $u\equiv 0$.
As a result, any invariance problem with a non-trivial solution will not satisfy this sufficient condition.
It follows that the weaker condition $\|u_n(\cdot,A)\|<1$, for some $n\geq 1$,
is also sufficient for the uniqueness of the solution of \eqref{eq:Bellman-u}. Let us introduce the quantities
\begin{equation*}
    m(A) := \inf\left\{m\geq0:\|u_m(\cdot,A)\|<1\right\}, \quad \rho(A) := \left\|u_{m(A)}(\cdot,A)\right\|,
\end{equation*}
for any $A\in \b$,
where we set $\rho(A) := 1$ if $m(A) = \infty$.
Note that both $m$ and $\rho$ are monotone functions on $\b$,
i.e. if $A\subseteq B$ are measurable sets,
then $m(A)\leq m(B)$ and $\rho(A)\leq \rho(B)$.
The quantity $m(A)$ is discussed in more detail for the special case of Markov Chains in Section \ref{ssec:integral}.

\begin{proposition}\label{prop:cmp}
    Let $A\in \b$ and denote for simplicity $m:=m(A)$ and $\rho := \rho(A)$. Then:
    \begin{itemize}
      \item[i.] if $m<\infty$,
      then $u(\cdot;A)\equiv 0$,
      and for all $n\geq 0$ it holds that $\|u_n(\cdot;A)\|\leq \rho^{\left\lfloor\frac{n}{m}\right\rfloor}$;
      \item[ii.] if $A,B\in\b$ are disjoint\footnote{In the following, for the sake of the simplicity the set-valued arguments of the reach-avoid value functions are assumed to be disjoint.
      This assumption does not affect the generality of the results, since $w(x;A,B) = w(x;A\setminus B,B)$ and hence any reach-avoid problem can be always considered as a problem on disjoint sets.} and $m<\infty$, then for all $n\geq 0$
        \begin{equation}\label{eq:w-bounds}
              0\leq w(x;A,B) - w_n(x;A,B)\leq \frac{m}{1-\rho}\rho^{\left\lfloor\frac{n}{m}\right\rfloor}.
        \end{equation}
    \end{itemize}
\end{proposition}

\begin{proof}
    For part (i), we have from \eqref{eq:DP-u} that $u_n = (\i_A)^{n-k}u_k$, for all $0\leq k\leq n$.
    Clearly, from the finiteness of $m$ and the definition of $\rho$ it follows that $u_{m}(\cdot;A)\leq 1_A(\cdot)\rho$, so
    \begin{equation*}
        u_n(\cdot;A) \leq \rho\cdot(\i_A)^{n-m}1_A(\cdot) = \rho u_{n-m}(\cdot;A).
    \end{equation*}
    for $n\geq m$. By induction we obtain that $\|u_n(\cdot;A)\|\leq \rho^{\left\lfloor\frac{n}{m}\right\rfloor}$,
    so that
    \begin{equation*}
      u(\cdot;A) = \lim\limits_{n\to\infty}u_n(\cdot;A) = 0.
    \end{equation*}

    For part (ii), we define functions $\Delta_n(x) := w_{n+1}(x;A,B) - w_n(x;A,B)$. Clearly, it holds that $\Delta_0(x) = 1_A(x)P(x,B)$ and $\Delta_{n+1}(x) = \i_A\Delta_n(x)$. Moreover, from the fact that $\Delta_0(x)\leq u_0(x;A)$ and the monotonicity of the operator $\i_{A}$, we have that $\Delta_n(x)\leq u_n(x;A)$. It further follows that
    \begin{equation*}
        w(x;A,B) - w_n(x;A,B) = \sum\limits_{i=n}^\infty \Delta_i(x) \leq \sum\limits_{i=n}^\infty \rho^{\left\lfloor\frac nm\right\rfloor} \leq \sum\limits_{k=\left\lfloor n/m\right\rfloor}^\infty m\rho^k = \frac{m}{1-\rho}\rho^{\left\lfloor\frac{n}{m}\right\rfloor},
    \end{equation*}
    as desired.
\end{proof}

As mentioned before,
one goal of this section is to reduce a given infinite-horizon problem to a finite-horizon one,
with the ability to tune the error incurred in this reduction.
If $m(A)<\infty$,
and since the right-hand side in \eqref{eq:w-bounds} decreases exponentially fast with respect to $n$,
Proposition \ref{prop:cmp} provides a method to achieve this.
In the following, the condition $m(\cdot)<\infty$, for an appropriate set-valued argument, indicates that the corresponding infinite-horizon problem can be reduced (and thus solved).

It is worth mentioning that Proposition \ref{prop:cmp} elucidates the difficulty in the direct extension of the error bounds in \cite{aklp2010,sa2011} from finite- to infinite-horizon problems: the developed finite-horizon approximation techniques can be interpreted as providing a perturbation $\tilde P$ of the original stochastic kernel $P$. Thus, they are
tailored at rendering the one-step error $\|\tilde  P  -  P \|$ (under the operator norm in \eqref{eq:operator-norm}) as small as possible.
However, in general a bound on the one-step error cannot be extended over an infinite time horizon, as the following argument shows.
Let us consider the case where the solution of the invariance problem on a set $A$ for the \dt-MP $(X,\b,P)$ is non-trivial.
We denote the corresponding value function as $u(x;A)$.
It follows from Remark \ref{rem:least-fixpoint} that $\|u\| = 1$.
Let $\nu$ be any probability measure on $(X,\b)$ such that $\nu(A^c)>0$,
and define $P^\delta(x;\cdot) := (1-\delta)P(x;\cdot)+\delta \nu(\cdot)$, for $\delta\in (0,1)$. We have
\begin{equation*}
    \| P ^\delta f -  P  f\| = \left\| \delta \cdot \int_{X} f(y)\nu(\d y) -  \delta\cdot   P  f\right\|\leq \delta \left(\|f\| + \| P  f\|\right)\leq 2\delta \|f\|,
\end{equation*}
for any function $f\in \B$. Hence $\| P ^{\delta} -  P \|\leq 2\delta$, so that it can be made arbitrarily small.
On the other hand, if we denote by $u^\delta$ the solution of the invariance problem on $A$ for \dt-MP $(X,\b,\tilde P^\delta)$,
we obtain that $\|u^\delta_1\|\leq 1-\delta\cdot\nu(A^c)<1$.
As a result, $u^\delta \equiv 0$ by Proposition \ref{prop:cmp}, so that $\|u - u^\delta\| = 1$,
regardless of how small $\delta$ is.

\subsection{Absorbing and simple sets}\label{ssec:abs}

From Proposition \ref{prop:cmp} it follows that the condition $m(A)<\infty$ in particular implies the uniqueness of the solution of the corresponding Bellman equation.
It turns out that under some continuity assumptions on the kernel $P$ this condition is also necessary.
Before we proceed, we introduce the notion of absorbing set, which is crucial for further discussions.

\begin{definition}
A set $A\in \b$ is called absorbing if $P(x,A) = 1$, for all $x\in A$. If for $A\in \b$ there is an absorbing subset $A'\subseteq A$ such that $A''\subseteq A'$ whenever $A''\subseteq A$ is absorbing, then we say that $A'$ is the largest absorbing subset of $A$ and write $A' = \las(A)$. The set $A$ is called simple if it does not have non-empty absorbing subsets, i.e. $\las(A) = \emptyset$, and non-simple otherwise.
\end{definition}

Clearly, the whole state space $X$ and the empty set $\emptyset$ are always absorbing,
and if $(A_n)_{n\geq 0}$ is a countable sequence of non-empty absorbing sets,
then their union $\bigcup_n A_n$ is absorbing and non-empty.
However, it is by no means clear that $\las(A)$ exists for any given set $A$,
since $A$ may contain uncountably many absorbing subsets and their union may not be even measurable.
Surprisingly, invariance value functions are useful to show that $\las(A)$ is always well-defined.

\begin{lemma}
\label{lem:A_n}
    Let $A\in \b$ and denote $A_n := \left\{u_n(\cdot;A) = 1\right\}$ for all $n\geq 0$,
    so that $A_0 = A$. Further, let $A_\infty := \bigcap_{n=0}^\infty A_n\in\b$, then for all $n\geq 0$ it holds that $A_{n+1}\subseteq A_n$ and
    \begin{equation}
    \label{eq:A_n}
        A_{n+1} = \left\{x\in A: P(x,A_n) = 1\right\}.
    \end{equation}
    The set $A_\infty$ admits the representation $A_\infty = \left\{u(\cdot;A) = 1\right\} = \las(A)$, i.e. it is the largest absorbing subset of $A$.
    In particular, if $m(A)<\infty$ then $A$ is simple.
\end{lemma}

\begin{proof}
    Let us first prove \eqref{eq:A_n}: for an arbitrary $x\in A_{n+1}$ it holds that
    \begin{equation}\label{eq:lem.proof}
        P(x,A)\leq 1=u_{n+1}(x;A) = 1_A(x)\int_X u_n(y;A)P(x,\d y) = \int_A u_n(y;A)P(x,\d y).
    \end{equation}
    Subtracting the right-hand side of \eqref{eq:lem.proof} from the left-hand side, we obtain that
    \begin{equation*}
      \int_A(1-u_n(y;A))P(x,\d y) = 0
    \end{equation*}
    as it is non-positive from \eqref{eq:lem.proof} and the integrand is non-negative. Due to the latter fact, we obtain that $P(x,\{u_n(\cdot;A) = 1\}) = 1$ or equivalently $P(x,A_n) = 1$.

    Conversely, let $x\in A$ be an arbitrary state that satisfies $P(x,A_n) = 1$. Let us show that $x\in A_{n+1}$. Indeed,
    \begin{equation*}
        u_{n+1}(x;A) = \int_X u_n(y;A)P(x,\d y)\geq \int_{A_n}u_n(y;A)P(x,\d y) = P(x,A_n) = 1,
    \end{equation*}
    thus $x\in A_{n+1}$. As a result, we have shown that \eqref{eq:A_n} holds true.

    Since $A_{n+1} = \{x\in A:P(x,A_n) = 1\}$ and $A_1 \subseteq A_0$, we obtain that $A_2\subseteq A_1$. Furthermore, by induction it holds that $A_{n+1}\subseteq A_n$ for all $n\geq 0$. If $u(x;A) = 1$ for some $x\in A$, then $u_n(x;A) = 1$ and $x\in A_n$ for all $n\geq 0$, hence $x\in A_\infty$.
    If $x\in A_\infty$, then $x\in A_n$ for all $n\geq 0$, so that $u(x;A) = \lim_{n\to\infty}u_n(x;A) = 1$.

    Suppose now that $A$ is a non-simple set and that $A'$ is an arbitrary absorbing subset of $A$.
    Clearly, it holds that $u(x;A) = 1$ for all $x\in A'$,
    hence $A'\subseteq A_\infty$.
    Furthermore,
    if $A_\infty\neq\emptyset$, then for any $x\in A_\infty$ and $n\geq 0$ it holds that $x\in A_{n+1}$, hence $P(x,A_n) = 1$. This implies that $A_\infty$ is absorbing since
    \begin{equation*}
        P(x,A_\infty) = P\left(x,\;\bigcap\limits_{n=0}^\infty A_n\right) = \lim\limits_{n\to\infty}P(x,A_n) = 1,
    \end{equation*}
    which leads to conclude that $A_\infty$ is the largest absorbing subset of $A$.
\end{proof}

As it has been mentioned above,
some continuity assumptions on the kernel $P$ are needed in order to sharpen the results.
To do so, the state space needs to be endowed with a certain topological structure (see e.g. \cite{hll1996}).

\begin{definition}
A state space $(X,\b)$ is called topological if $X$ is a Borel subset of a Polish (i.e. a metrizable, complete, and separable) space and if $\b$ is a Borel $\sigma$-algebra of $X$. A kernel $P$ on a topological space is called weakly continuous (or Feller) if the function $ P  f$ is upper semi-continuous (u.s.c.) whenever $f\in \B$ is u.s.c. \cite[Appendix C]{hll1996}.

A \dt-MP $(X,\b,P)$ is said to be weakly continuous whenever $(X,\b)$ is a topological state space and $P$ is weakly continuous.
\end{definition}

The next theorem shows that for a weakly continuous \dt-MP,
the infinite-horizon problem over a compact set $A$ can be directly reduced to the finite-horizon one (in the sense that $m(A)< \infty$) if and only if the set $A$ is simple.

\begin{theorem}\label{thm:main}
    Let $(X,\b)$ be a topological state space and $A$ be a compact set. If $P$ is weakly continuous then $\las(A)$ is a compact set and the following statements are equivalent:
    \begin{enumerate}
      \item $m(A)<\infty$;
      \item $\i^n_A$ is a contraction on $\B$ for some finite $n$ (contractivity);
      \item equation \eqref{eq:Bellman-u} has a unique solution (uniqueness);
      \item $u(x;A)  = 0$ for all $x\in X$ (triviality);
      \item the set $A$ is simple: $A_\infty = \emptyset$ (simplicity).
    \end{enumerate}
\end{theorem}
\begin{proof}
    \textit{1)} $\Rightarrow$ \textit{2)} Clearly, for any function $f\in \B$ it follows that $\i_A f(x)\leq \|f\|1_A(x)$ for all states $x\in X$. Thus if $m(A)<\infty$, then $\i^{m(A)+1}$ is a contraction since
    \begin{equation*}
    \|\i^{m(A)+1}_A f\| \leq \|f\|\cdot\|\i^{m(A)}_A1_A\| = \|f\|\cdot\|u_{m(A)}(\cdot;A)\| \leq \rho(A)\|f\|.
    \end{equation*}

    \textit{2)} $\Rightarrow$ \textit{3)} If $f\in \B$ be a solution of \eqref{eq:Bellman-u}, i.e. $f = \i_A f$. By induction we have $f = \i_A^n f$, which by contraction mapping theorem \cite[Proposition A.1]{hl1989} implies the uniqueness of the fixpoint $f$.

    \textit{3)} $\Rightarrow$ \textit{4)} follows from the linearity of \eqref{eq:Bellman-u} and \textit{4)} $\Rightarrow$ \textit{5)} from Lemma \ref{lem:A_n}, so we only have to show that \textit{5) }$\Rightarrow$ \textit{1)}. Suppose this is not true, i.e. $m(A) = \infty$ but $A$ is simple. It follows that $A_n\neq\emptyset$ for all $n\geq 0$. Since $A$ is compact and $X$ is metrizable, $A$ is closed and hence $u_0 = 1_A$ is u.s.c. Hence $u_n$ is u.s.c. for all $n\geq 0$ by the weak continuity of $P$, which implies that all sets $A_n = \{u_n(\cdot;A)\geq 1\}$ are compact. Moreover, they are not empty and so their intersection $A_\infty$ is compact and non-empty, which leads to a contradiction.
\end{proof}

\begin{remark}\label{rem:thm.main}
 Within the main goal of reducing infinite-horizon problems over a set $A$ to finite-horizon ones,
 let us remark that numerical methods for finite-horizon problems leading to the computation of PCTL value functions with any given precision have been developed, up to our knowledge, only for compact subsets of finite-dimensional metric spaces \cite{aklp2010,sa2011} -- this aligns with the assumption raised for Theorem \ref{thm:main}. Also, conditions required on the kernel $P$ in loc. cit. are stronger than the weak continuity raised above. Taking all of this into account, the assumptions in Theorem \ref{thm:main} are rather mild. Furthermore, some of the relations in the theorem are true under even weaker conditions: we postpone the discussion of these facts to Section \ref{sec:appendix} (Appendix).
\end{remark}

\begin{remark}
  It follows directly from Theorem \ref{thm:main} that if $m(A)<\infty$
  then $\i^{m(A)+1}_A$ is a contraction and furthermore, $\|\i^{m(A)+1}_A\| \leq \rho(A)$.
\end{remark}

\subsection{A decomposition technique}\label{ssec:decomp}

Although Theorem \ref{thm:main} is stated in terms of value functions for the invariance problem,
its application to the issue of uniqueness of the solution of a reach-avoid problem is direct,
since \eqref{eq:Bellman-u} is a homogeneous version of \eqref{eq:Bellman-w}.
As a result, if the \dt-MP $(X,\b,P)$ is weakly continuous, sets $A$, $B$ are disjoint and $A$ is compact and simple, then $m(A)<\infty$ and the reach-avoid problem can be solved.
Thus, the next goal is to study the case of a non-simple set $A$.
For this objective the characterization given in Theorem \ref{thm:main} is again useful.
We proceed assuming that the $\las$ of a given set is known,
and leave the discussion on the characterization of the $\las$ of a given set and the verification of the simplicity of a set to Section \ref{ssec:v.o.s.}.

If $A$ is non-simple,
the main issue preventing an efficient solution of the problem is the presence of an absorbing subset $\las(A)$.
This leads to the lack of contractivity of the operator $\i_A$ and to the non-uniqueness of the solution of \eqref{eq:Bellman-w}.
Intuitively,
if we were to remove some neighborhood $C \supset \las(A)$,
then we would expect that $m(A\setminus C)<\infty$,
so that a related problem can be solved on $A\setminus C$.
Moreover, recall that the solution of the original problem on $\las(A)$ is known: $w(x;A,B) = 0$ for all $x\in \las(A)$,
since such states initialize trajectories that never reach the set $B$.
The following result relates the solutions of the two problems:

\begin{lemma}[Decomposition technique]\label{lem:decomp}
    Let sets $A,B\in\b$ be disjoint,
    and let the set $C\in \b, C \subseteq A$ be such that the invariance value function $u(\cdot;A\setminus C)\equiv 0$.
    Then $w(x;A\setminus C,B)$ is the unique solution of the corresponding Bellman equation
    \begin{equation}\label{eq:decompBel}
      w(x;A\setminus C,B) = 1_B(x)+\i_{A\setminus C} w(x;A\setminus C,B),
    \end{equation}
    and for all $x\in X$ the following holds:
    \begin{equation}\label{eq:decomp}
      0\leq w(x;A,B) - w(x;A\setminus C,B)\leq \sup\limits_{y\in C}w(y;A,B).
    \end{equation}
\end{lemma}

\begin{proof}
  As an abbreviation,
  let us denote $\tau_1 = \tau_{B\cup C}$ and $\tau_2 = \tau_{(A\cup B)^c}$ and let us partition the event space $\Omega$ by the following four disjoint hypotheses\footnote{
    Alternatively, these hypotheses can be defined using PCTL framework,
    and are given by path formulae
    $H_1 = \neg(B\vee C)\mathsf U \neg (A\vee B)$,
    $H_2 = \square (A\setminus C)$,
    $H_3 = (A\setminus C)\mathsf U B$ and
    $H_4 = (A\setminus C)\mathsf U C$.
  }:
  \begin{equation*}
    H_1 := \{\tau_2<\tau_1,\tau_2<\infty\},\quad H_2 := \{\tau_1 = \infty,\tau_2 = \infty\},
  \end{equation*}
  \begin{equation*}
    H_3 := \{\tau_1<\tau_2,\tau_1<\infty,\tau_1 =\tau_B\},\quad H_4 := \{\tau_1<\tau_2,\tau_1<\infty,\tau_1 =\tau_C\}.
  \end{equation*}
  Recall that $w(x;A,B) =  \pr_x\{\tau_B<\tau_2,\tau_B<\infty\}$, thus
  \begin{equation*}
    w(x;A,B) = \sum\limits_{i=1}^4 \pr_x\left(\left\{\tau_B<\tau_2,\tau_B<\infty\right\}\cap H_i\right).
  \end{equation*}
  Note that the first term is zero since clearly $\{\tau_B<\tau_2,\tau_B<\infty\}\cap H_1 = \emptyset$.
  The second term vanishes because $ \pr_x(H_2) = u\left(x;A\setminus C\right)\equiv 0$.
  Since $H_3\subseteq\left\{\tau_B<\tau_2,\tau_B<\infty\right\}$ the third term equals to $ \pr_x(H_3) = w(x;A\setminus C,B)$,
  which leaves only the fourth term to be studied.
  Let $x$ be any such that $ \pr_x(H_4)\neq 0$ and define a measure $\nu_x$ on $(X,\b)$ by
  \begin{equation*}
    \nu_x(D) :=  \pr_x\left(\left.\x_{\tau_C}\in D\right|H_4\right),\quad D\in \b,
  \end{equation*}
  so that clearly $\nu_x(C^c) = 0$. For such fixed $x$ it holds that
  \begin{align*}
    0\leq  \pr_x\left(\left\{\tau_B<\tau_2,\tau_B<\infty\right\}\cap H_4\right) &=  \pr_x\left(\left.\left\{\tau_B<\tau_2,\tau_B<\infty\right\}\right| H_4\right) \pr_x(H_4)
    \\
    &= \left(\int_C w(y;A,B)\nu_x(\d y)\right)\cdot w(x;A,C)
    \\
    &\leq \sup\limits_{y\in C}w(x;A,B).
  \end{align*}
  The same bounds clearly hold in the alternative case $ \pr_x(H_4) = 0$.

  Finally, it follows that $w(x;A\setminus C,B)$ is the unique solution of the corresponding Bellman equation \eqref{eq:decompBel} from $u(\cdot;A\setminus C)\equiv 0$ (see Proposition \ref{prop:uniqueness} in Section \ref{sec:appendix}).
\end{proof}

\begin{corollary}\label{cor:decomp}[From Lemma \ref{lem:decomp}]
    Let $(X,\b,P)$ be a weakly continuous \dt-MP and let $A,B\in \b$ be disjoint and such that $A$ is a compact, non-simple set. Let $C\subseteq A$ be an open neighborhood of $\las(A)$ in the subspace topology of $A$. Then \eqref{eq:decomp} holds for all $x\in X$, and $m(A\setminus C)<\infty$.
\end{corollary}

\begin{proof}
    Since $C$ is open in $A$,
    the set $A\setminus C$ is a closed subset of a compact set $A$ and thus itself compact.
    From the inclusions $\las(A)\subset C\subset A$ it follows that $A\setminus C$ is simple,
    hence Theorem \ref{thm:main} ensures that all the conditions of Lemma \ref{lem:decomp} are satisfied.
\end{proof}

\smallskip

In order to render the result in Corollary \ref{cor:decomp} useful for the computation of the infinite-horizon reach-avoid value function,
we should provide a method to choose an open neighborhood $C$ of $\las(A)$,
such that $\sup_{y\in C}w(y;A,B)<\ve$,
where $\ve>0$ is a given precision level.
We use the theory of excessive functions \cite{s2008} to achieve this goal.

\begin{definition}
    Given a function $g\in\B$, the excessive set of $g$ is $\e_g = \left\{ P  g- g \leq 0\right\}$.
    If $\e_g = X$,  i.e. if $ P  g(x)\leq g(x)$ for all $x\in X$,
    then the function $g\in \B$ is called excessive.
\end{definition}

The relation between excessive functions and infinite-horizon invariance is given via Doob's inequality \cite{s2008}:
if $g\in \B$ is an excessive,
non-negative function,
then
\begin{equation}\label{eq:doob}
   \pr_x\left\{\sup\limits_{n\geq 0}g(\x_n)\geq \delta\right\}\leq \frac{g(x)}{\delta}.
\end{equation}
for all $\delta>0$. The inequality \eqref{eq:doob} can be rewritten via the invariance value function:
\begin{equation}\label{eq:u-doob}
    u\left(x;\left\{g<\delta\right\}\right)\geq 1 - \frac{g(x)}{\delta}.
\end{equation}

Excessive functions for stochastic systems are akin to Lyapunov functions for deterministic systems, since they are characterized by decreasing behavior along the dynamics of the process,
as the inequality $ P  g\leq g$ suggests\footnote{From the definition of $ P $ is follows that $ P  g(x) = \mathsf E_x[g(\x_1)]$, where $\mathsf E_x$ denotes the expectation with respect to $ \pr_x$. Thus the condition $ P  g\leq g$ means that the expected value of the function $g$ at the next time step is bounded by its current value,
so that the function $g$ does not increase on average along realizations of the \dt-MP. Thus the function $g\in \B$ is excessive if and only if the process $(g(\x_n))_{n\geq 0}$ is a $ \pr_x$-supermartingale, for all $x\in X$ \cite[p.20]{ps2006}.}.
As is the case with Lyapunov functions for deterministic systems,
it is non trivial to find excessive functions.
However, it is possible to relax the assumption on global excessivity and to employ a local version of Doob's inequality.

\begin{lemma}\cite[Theorem 12]{k1967}\label{lem:LDI}
  Let $g\in \B$ be a non-negative function such that for some $\delta>0$ it holds that $\{g<\delta\}\subseteq \mathcal{E}_g$.
  Whenever $x\in \{g<\delta\}$,
  it follows that
  \begin{equation}
  \label{eq:doob.kushner}
     \pr_x\left\{\sup\limits_{n\geq 0}g(\x_n)\geq \delta\right\}\leq \frac{g(x)}{\delta}.
  \end{equation}
\end{lemma}

The idea behind the proof of this lemma is to consider a set $A = \{g<\delta\}$.
The related invariance value function does not depend on $P(x,\cdot)$ for $x\in A^c$,
where it is simply equal to zero (recall that all the integrals in the DP recursion \eqref{eq:DP-u} are equivalently taken over the set $A$).
As a result, exclusively the dynamics
within the set $A$ are important for the process.
Note also that for $x\notin \{g<\delta\}$ one trivially has \eqref{eq:doob.kushner} as in such case the term in the right-hand side is greater or equal than $1$.

\begin{definition}
     For a topological state space $(X,\b)$ we say that a non-negative continuous function $g\in \B$ is $\delta$-locally excessive on the set $A\in \b$ if $\{g = 0\} = \las(A)$ and for some real number $\delta>0$ it holds that $\{g<\delta\}\subseteq A$ and $\{g<\delta\}\subseteq \e_g$.
\end{definition}

\begin{theorem}\label{thm:decomp}
  Let $(X,\b,P)$ be a weakly continuous \dt-MP and let $A,B\in \b$ be disjoint and such that $A$ is a compact, non-simple set. If there exists a $\delta$-locally excessive function $g$ on $A$, then for any $\ve\in (0,1)$ it holds that $m(A\setminus \{g<\ve\delta\})<\infty$ and that
  \begin{equation}\label{eq:decomp-Lyap}
    0\leq w(x;A,B) - w\left(x;A\setminus\{g<\ve\delta\},B\right)\leq \ve.
  \end{equation}
\end{theorem}

\begin{proof}
  First, we show that for any $\ve\in (0,1)$,
  if $g(x)<\ve\delta$ then $w(x;A,B)\leq \ve$.
  Indeed, as $\{g<\delta\}\subseteq \e_g$, by Lemma \ref{lem:LDI} we have that $\displaystyle{u(x;\{g<\delta\})\geq 1-\frac{g(x)}{\delta}}$ for all $x\in X$, so
  \begin{equation*}
    u(x;A)\geq 1-\frac{g(x)}{\delta}
  \end{equation*}
  for all $x\in X$, which follows from $\{g<\delta\} \subseteq A$. Since $u(x;A) = 1-w\left(x;X,A^c\right)$, then
  \begin{equation*}
    w(x;X,A^c)\leq \frac{g(x)}{\delta}
  \end{equation*}
  for all $x\in X$, and since $A\subseteq X$ and $B\subseteq A^c$,
  from Proposition \ref{prop:subsat} it follows that $w(x;A,B)\leq\frac{g(x)}{\delta}$.
  As a result, for any $x\in \{g<\ve\delta\}$ it holds that $w(x;A,B)\leq\ve$.

  Second, let us fix any $\ve\in(0,1)$ and denote $C = \{g<\ve\delta\}$. Clearly, $\las(A)\subseteq C$ and the set $\{g<\ve\delta\}$ is open in $A$ since $g$ is continuous on $A$. The statement of the theorem then follows from Corollary \ref{cor:decomp}.
\end{proof}

\subsection{Integral kernels and discrete-space Markov Chains
}\label{ssec:integral}

From Theorem \ref{thm:decomp} it follows that for weakly continuous \dt-MPs a reach-avoid problem on a non-simple set can be solved if an appropriate locally excessive function is found.
For a known and studied subclass of these processes the problem can be solved even without resorting to such functions.
We write $P(x,\d y) = p(x,y)\mu(\d y)$ if $P$ is an integral kernel with a basis $\mu$ and a density $p$,
namely when $\mu$ is a $\sigma$-finite non-negative measure, the function
$p:X\times X\to [0,\infty)$ is jointly measurable, and for any $A\in \b$ it holds that
\begin{equation*}
    P(x,A) = \int_{A}p(x,y)\mu(\d y).
\end{equation*}
Furthermore, we raise the following assumption,
which generalizes the one used for related studies over the finite horizon \cite{aklp2010,sa2011}.

\begin{assumption}\label{as:1}
  For a subset $A\in \b$ of a topological state space $(X,\b)$ assume that the function $P(\cdot,D)$ is continuous on $A$ for any $D\subseteq A$, $D\in \b$.
\end{assumption}

Let us mention some sufficient conditions for the Assumption \ref{as:1} to hold true for integral kernels. It follows from \cite[Example C.6]{hll1996} that whenever $p(\cdot,y)$ is a continuous function on the set $A\in \b$ for all $y\in A$, then Assumption \ref{as:1} is satisfied for the set $A$. It is thus milder than aforementioned assumptions of \cite{aklp2010,sa2011} where the stronger Lipschitz continuity is required instead.

Before we prove the main result, we need the following lemma that connects the condition $m(\cdot)<\infty$ with an important notion of the uniform transitivity \cite{mt1993}.

\begin{lemma}\label{lem:ut}
  Let $(X,\b,P)$ be a \dt-MP. Suppose that the set $A\in \b$ is uniformly transient, i.e. there exists $M < \infty$ such that $\sum_{n=0}^\infty  P ^n 1_A(x) \leq M$ for all $x\in A$. Then $m(A) < \infty$.
\end{lemma}

\begin{proof}
  Suppose that $m(A) = \infty$, then for any $n\in \N_0$ there exists a point $x_n\in A$ such that $u_n(x_n;A)\geq \frac12$. Clearly, for any $0\leq k\leq n$ it further holds that
  \begin{equation*}
    u_k(x_n;A)\geq u_n(x_n;A)\geq \frac12.
  \end{equation*}
  Note, that for any non-negative function $f \in \B$ and for any $n\in \N_0$ it holds that $\i_A^n f(x)\leq  P ^n f(x)$ for all $x\in X$. As a result:
  \begin{equation*}
    \sum_{n=0}^\infty  P ^n 1_A(x) \geq \sum_{n=0}^\infty \i_A^n 1_A(a) = \sum_{n=0}^\infty u_n(x;A).
  \end{equation*}
  On the other hand, $\sum_{k=0}^\infty u_k(x_n;A) \geq \frac n2$ and thus $A$ is not uniformly transient.
\end{proof}

We are now ready to state the main result of the section.

\begin{theorem}\label{thm:Bellman-w.mod}
    Let $(X,\b,P)$ be a \dt-MP on a topological state space $(X,\b)$ and suppose that $A\in \b$ is a compact set satisfying Assumption \ref{as:1}.
    For any set $B\in \b$ disjoint from $A$ it holds that $m(A\setminus \las(A))<\infty$ and for all $x\in X$
    \begin{equation}\label{eq:Bellman-w.mod}
        w(x;A,B) = w(x;A\setminus\las(A),B).
    \end{equation}
\end{theorem}

\begin{proof}
  First of all, in case $m(A\setminus \las(A)) <\infty$ the equality \eqref{eq:Bellman-w.mod} follows immediately from Lemma \ref{lem:decomp}, in particular from \eqref{eq:decomp} with $C = \las(A)$. To show that $m(A\setminus \las(A)) <\infty$ we apply the Doeblin decomposition of $A$ into a finite number of absorbing sets and a uniformly transient set \cite{tt1979}.

  Fix some point $a\notin A$ and define a new \dt-MP $(X',\b',P')$ where $X' = A\cup \{a\}$ endowed with a disjoint union topology \cite{r1976} and $\b'$ is its Borel $\sigma$-algebra. Define a kernel $P'$ by the formulae $P'(a,\{a\}) = 1$, $P'(x,\{a\}) = P(x,A^c)$ and $P'(x,D) = P(x,D)$ for any set $D\in \b'$ such that $D\subseteq A$. It clearly holds that $P(\cdot,D)$ is a continuous function on $X'$ for any $D\in \b'$ and thus from \cite[Theorem 7.1]{tt1979} it follows that $A$ can be represented as a disjoint union
  \begin{equation*}
    A = H \cup E,
  \end{equation*}
  where $H$ is absorbing and $E$ is uniformly transient with respect to the kernel $P'$. Since $H\subseteq A$, $E\subseteq A$, and thanks to the fact that kernels $P$ and $P'$ agree on $A$, we obtain that $H$ is absorbing and $E$ is uniformly transient with respect to the original kernel $P$. As a result, $H\subseteq \las(A)$ and $m(E)<\infty$ by Lemma \ref{lem:ut}, which further leads to the fact that
  \begin{equation*}
    m(A\setminus \las(A)) \leq m(A\setminus H) = m(E)<\infty,
  \end{equation*}
  as desired.
\end{proof}

%
%

\begin{remark}
  In the special case of the invariance problem over the set $A$,
  under the assumptions of Theorem \ref{thm:Bellman-w.mod} it follows that $u(x;A) = w(x;A,\las(A))$.
  Moreover, the proof of Lemma \ref{lem:decomp} implies that for any initial state $x\in X$,
  $ \pr_x$-a.s. a trajectory $(\x_n)_{n\geq 0}$ of the \dt-MP that stays invariant in the set $A$ necessarily reaches its largest absorbing subset.
  Altogether, this enlightens yet another interesting relation between invariance and reach-avoid problems.
\end{remark}

\medskip

Let us now elucidate the meaning of the results obtained above by considering a special case where the state space is countable,
i.e. when the process is a discrete-time Markov Chain (\dt-MC).
The methods we developed are directly applicable to the \dt-MC framework:
they generalize those presented in \cite[Section 10.1.1]{bk2008}.

Without loss of generality,
let us assume that $X = \N$ is the state space of a \dt-MC.
We endow $X$ with the discrete metric $d(i,j) = 1\{i\neq j\}$, so that $\b = 2^X$.
The basis $\sigma$-finite measure is chosen to be the counting one:
$\mu(\{i\}) = 1$, for any $i\in X$.
Any stochastic kernel $P$ over $(X,\b)$ can be expressed as a  matrix $\mathrm P = (p_{ij})_{i,j\in \N}$,
where $p_{ij}:=P(i,\{j\})$.
With the chosen counting measure,
the entries of the stochastic matrix $\mathrm P$ determine the density function,
namely $p(i,j) = p_{ij}$.
Recall that in the discrete topology compact sets are exactly finite sets.
Thus, any finite set $A\subseteq X$ satisfies Assumption \ref{as:1} since on the discrete topological space any function is continuous.

\begin{remark}\label{rem:las-dtMC}
    For a \dt-MC the largest absorbing subset of any finite set can be found algorithmically. Indeed, from Lemma \ref{lem:A_n} it follows that $\las(A) = \{u(\cdot;A) = 1\}$, so the set can be equivalently expressed via a CTL formula: $\las(A) = \left\{x\in A:x\models \forall\square A\right\}$.
    As such, it can be computed in $\mathscr O(\mu^2(A))$ time over the adjacent graph of the original \dt-MC (to be defined below) \cite[Theorem 6.30]{bk2008}.
\end{remark}

\begin{corollary}[from Theorem \ref{thm:Bellman-w.mod}]
    Let sets $A$, $B\in \b$ be disjoint and let the set $A$ be finite. Denote $\tilde A :=A\setminus \las(A)$ and $b_i:=P(i,B) = \sum_{j\in B}p_{ij}$. The reach-avoid value function $w(i;A,B)$ is defined uniquely as a solution of the system of linear equations
    \begin{equation}\label{eq:Bellman-w.dtMC}
    \begin{cases}
      w(i;A,B) = 1&\text{ if }i\in B,
      \\
      w(i;A,B) = b_i+\sum\limits_{j\in \tilde A} p_{ij} w(j;A,B) &\text{ if }i\in \tilde A,
      \\
      w(i;A,B) = 0&\text{ otherwise.}
    \end{cases}
    \end{equation}
\end{corollary}

\begin{proof}
    As it is mentioned above Assumption \ref{as:1} is satisfied. It further follow from Theorem \ref{thm:Bellman-w.mod} that \eqref{eq:Bellman-w.mod} holds true and that the corresponding Bellman equation has the form $ w(i;A,B) = 1_B(i)+1_{\tilde A}(i)\sum_{j\in \N}p_{ij}w(j;A,B)$,
    which is equivalent to \eqref{eq:Bellman-w.dtMC}.
\end{proof}

Note, that to find a solution for \eqref{eq:Bellman-w.dtMC},
one should solve a system of linear equations with a non-zero determinant.
Moreover, notice that the square submatrix $\tilde{\mathrm P} := (p_{ij})_{i,j\in \tilde A}$ in \eqref{eq:Bellman-w.dtMC} is contractive since $m(\tilde A)<\infty$,
so even for large-scale problems efficient numerical methods can be applied to solve the problem with any given precision.

Let us mention what the condition $m(A)<\infty$ means in graph-theoretical terms for a \dt-MC.
The adjacency graph of a \dt-MC is a directed graph $(V,E)$,
where with $V = X$ and the set of edges $E$ is such that $(i,j)\in E$ if and only if $p_{ij}>0$. To an arbitrary element $i\in A$ we can assign a positive number $m_i$, which is the length of the shortest path in the graph $(V,E)$ from $i$ to $A^c$. Clearly, it holds that $m(A) = \sup_{i\in A}m_i$. Moreover, from this characterization it can be easily seen that $m(A)\leq \mu(A)$ if $m(A)$ is finite. As a result, if $\|u_{\mu(A)+1}(\cdot;A)\| = 1$ it follows that $m(A) = \infty$ and that
\begin{equation*}
  \las(A) = A_{\mu(A)+1} = \left\{u_{\mu(A)+1}(\cdot;A) = 1\right\}.
\end{equation*}

\subsection{Verification of simplicity of $A$ and characterization of $\las(A)$}\label{ssec:v.o.s.}

Let us summarize the methods developed for the solution of the infinite-horizon reach-avoid problem in the previous sections.
Assume that $A,B\in \b$ are disjoint and let us focus on the case when the set $A$ is compact and the kernel $P$ is weakly continuous.
If $A$ is simple, it follows from Theorem \ref{thm:main} that the solution of \eqref{eq:Bellman-w} is unique and that $m(A)<\infty$ -- the solution can be found as in Proposition \ref{prop:cmp}.
If $A$ is non-simple, the solution of \eqref{eq:Bellman-w} is not unique and $m(A) = \infty$, thus Proposition \ref{prop:cmp} cannot be applied directly.
In the latter case,
there are two approaches to solve an infinite-horizon reach-avoid problem over a non-simple set $A$:
if Assumption \ref{as:1} holds true,
then Theorem \ref{thm:Bellman-w.mod} allows formulating an equivalent problem over the set $A\setminus \las(A)$.
Otherwise, one has to synthesize an appropriate $\delta$-locally excessive function to apply Theorem \ref{thm:decomp}.

All the instances discussed above depend on the fundamental issue of whether a given compact set $A$ is simple or not.
In general it is hard to provide an analytical answer to such a question,
and no known general automatic procedure enables computing absorbing sets exactly.
On the other hand,
the ``if and only if'' nature of the results in Theorem \ref{thm:main} implies that this issue is not a limitation that is specific to the techniques presented in this paper:
on the contrary, any other method aiming to solve a general infinite-horizon reach-avoid problem is bound to check the simplicity of a given set $A$.

Let us discuss instances of \dt-MP for which the $\las(A)$ of a given set $A$ can be found explicitly.
The case of \dt-MC, as discussed in Remark \ref{rem:las-dtMC},
has been recently extended to a subclass of \dt-MP with integral kernels in \cite[Chapter 4.2]{ta2012}.
In both instances all the conditions in Theorem \ref{thm:Bellman-w.mod} are satisfied,
thus the reach-avoid problem can be solved.

Given additional knowledge on the structure of a \dt-MP,
it may be easier to verify the simplicity of a given set $A$:
if $P$ is $\varphi$-irreducible \cite[Chapter 4]{mt1993}, then $A$ is simple whenever $\varphi(A^c)>0$. If $\varphi$ is the maximal irreducibility measure, then $A$ is simple if and \textit{only if} $\varphi(A^c)>0$. However, notice that for a given \dt-MP the  verification of its irreducibility can represent an even harder requirement than the verification of the simplicity of a specific set $A$. Moreover, observe that any \dt-MP admitting two disjoint  non-empty absorbing sets is not irreducible,
which points out the conservatism of this condition.

An additional example where further knowledge on the structure of \dt-MP may shed light on the absorbance of its sets
is provided in \cite{akm2011}.
As already mentioned, an automaton specification $\mathscr A$ over a \dt-MP $\mathscr H = (X,\b,P)$ can be verified as a reachability specification over the product $\mathscr A\times\mathscr H$, which is again a \dt-MP. The discrete structure of the automaton $\mathscr A$ can be exploited in order to determine absorbing sets within the product \dt-MP $\mathscr A\otimes\mathscr H$.

Furthermore, analytical methods can be applied to find absorbing sets.
If the dynamical system representation of a \dt-MP \eqref{eq:dynamical} is known one can try to characterize its absorbing sets,
as the examples of Section \ref{sec:cs} will display.
Also, for integral kernel $P(x,\d y) = p(x,y)\mu(\d y)$ with density $p$ given explicitly, one may try to check for simplicity using the following result.

\begin{proposition}\cite[Proposition 3]{ta2012}
    For $x\in X$ define $s(x) := \{y\in X:p(x,y)>0\}$. A set $A\in \b$ is absorbing if and only if $\mu(s(x)\setminus A) = 0$, for all $x\in A$.
\end{proposition}

Finally, although in general the verification of the simplicity of a given set is not a decidable procedure, the following method can be applied.
Let us consider the sequence $(A_n)_{n\geq 0}$ defined in Lemma \ref{lem:A_n}.
If $A_n = \emptyset$ for some $n\in \N$,
then clearly $A$ is simple.
Although the definition itself requires a precise characterization of $u_n(\cdot;A)$,
only the computation of $P(x,\cdot)$ is needed in \eqref{eq:A_n},
instead of consecutive integral iterations over value functions.
Let us now introduce an approximate approach for the computation,
using the concepts in Section \ref{ssec:nest}:
leveraging the procedure in \eqref{eq:A_n},
we have that $A_0 = A$ and that $A_{n+1} = \P_{\geq 1}[\Next A_n]$. Let us select a precision level $\delta\in (0,1)$ and construct a sequence of supersatisfaction sets as follows:
\begin{equation*}
    A^*_{n+1} = \P_{\geq 1-\delta}[\Next A^*_n], \quad A^*_0 = A.
\end{equation*}
By construction, $A_n\subseteq A^*_n$ for all $n\geq 0$, thus $A$ is simple whenever $A^*_n = \emptyset$ for some $n\in \N$.
Notice that the conditions required to implement the procedure are very general.
Let us discuss its strong and weak points:
\begin{itemize}
    \item If the exact form of $P$ is given, then the sets $A_n$ can be characterized explicitly.
    The simplicity of $A$ is verified if the sequence $(A_n)_{n\geq 0}$ eventually contains only empty sets.
    On the other hand,
    if $A$ is non-simple, then the set $\las(A)$ can be found whenever $A_n = A_{n+1}\neq \emptyset$ for some $n\in \N$.
    Clearly, in such a situation it holds that $\las(A) = A_n$.
    Finally, if $A$ is non-simple, whenever $A_{n+1}$ is a strict subset of $A_n$, one can compute $\las(A) = A_\infty$ as an intersection of the sets $A_n$.
    \item If only an approximate characterization of $P$ is available,
    the simplicity of set $A$ can be verified for sufficiently small $\delta$ and sufficiently large $n$.
    However,
    it is not clear how big $n$ should be taken to ensure that $A^*_n=\emptyset$ for a given precision level $\delta$.
    Due to this reason,
    it is extremely important to have an \textit{a priori} upper-bound on $m(A)$, provided the latter is finite (cfr. the discussion on $m(A)$ for \dt-MC in Section \ref{ssec:integral}).
    Furthermore,
    the non-simplicity of set $A$ cannot be verified:
    because of the errors in the computation of $A^*_n$,
    the case $A_n = A_{n+1}\neq\emptyset$ cannot be exactly characterized.
\end{itemize}

We conclude the discussion in this section with the following practical observation:
in practice stochastic kernels for a \dt-MP either are extracted from finite data coming from measurement experiments,
or derived from some underlying analytical model.
In the latter case, the model gives an additional knowledge on the structure of a \dt-MP which can be further used along the lines discussed in this section to find the largest absorbing subset of a given set or to verify the simplicity of such set.
Conversely, when no underlying model is known and kernels are interpolated exclusively from measurements data, any kernel resulted via an interpolation technique can be negligibly perturbed in order to yield absence of absorbing subsets of given compact sets (see discussion after Proposition \ref{prop:cmp}).

\subsection{Connections with the literature}
\label{ssec:connections}
Let us comment on the overall connection between the results achieved in this paper and related ones from classical literature on \dt-MP \cite{mt1993,n1984,r1984,tt1979}.
These works deal with similar problems and study related objects:
for instance
the DP recursions for functions $u_n$ can be obtained from the equations on ``taboo'' probabilities
\cite{mt1993}
and the invariance operator $\i_A$ can be related to the operator $P_A$ used in \cite{mt1993}.
However, the focus in this literature is not on the quantitative analysis of such objects,
but rather on the asymptotic behaviour of the underlying \dt-MP.
This literature deals for instance with the existence, uniqueness, and stability of invariant distributions:
although this is an important problem in the analysis of \dt-MP,
it does not allow for a direct connection with PCTL specifications.
Moreover, the strongest results in this literature are often obtained under the assumption of irreducibility,
which is both restrictive and hard to verify over a given \dt-MP.
Finally, these studies have not been concerned with computability issues,
so that the developed methods are rarely constructive:
for example, although the Doeblin decomposition of a compact set $A$ used in the proof of Theorem \ref{thm:Bellman-w.mod} may shed some light on the structure of $\las(A)$,
its construction is classically characterized by the infinite-horizon value functions \cite{tt1979},
which are here the objective of its use.

As a consequence of this discussion,
results in classical literature on \dt-MP do not appear to be directly applicable to the problems considered in this paper.
On the other hand, they may be useful in studying properties of absorbing sets that we showed are crucial for the approximate PCTL model-checking of general \dt-MP.
This connection represents a promising future direction of study, which however goes beyond the scope of the current contribution.

\section{Case studies}\label{sec:cs}

\subsection{A one-dimensional affine Gaussian system}\label{ssec:cs.1}

Let $X = \mathbb R$ be endowed with the standard topology and let $\b$ be its Borel $\sigma$-algebra.
Consider a sequence $(\xi_n)_{n\geq 0}$ of iid standard normal random variables and define a \dt-MP as
\begin{equation}\label{eq:example-dynamical}
   \x_{n+1} = (\alpha+\mu \x_n) + (\beta+\sigma \x_n)\cdot \xi_n,
\end{equation}
where $\alpha,\beta,\mu,\sigma \in \R$ are parameters and $\x_0 = x\in \R$.
In order to study the probabilistic invariance problem for this affine Gaussian model,
let us select a compact set $A$ in $\R$. Let us focus on how the structure of the dynamics are affected by the choice of the parameters.
In order to avoid trivial constant dynamics,
let us assume that at least one of the parameters $\alpha,\beta,\mu-1,\sigma$ is non-zero.
If $\beta+\sigma \x_n\neq 0$ the distribution of $\x_{n+1}$ admits the whole state space $\R$ as its support,
thus for $A$ to be non-simple it is necessary that point $\kappa:=-\frac{\beta}{\sigma} \in A$.
We then assume that $\sigma\neq 0$,
since clearly if $\sigma = 0$ any compact set is simple.
Moreover, for $A$ to be non-simple the state $\kappa$ has to be absorbing,
so from \eqref{eq:example-dynamical} it must hold that $\kappa = \alpha +\mu\kappa$,
so $\alpha = (1-\mu)\kappa$.
Since by Theorem \ref{thm:main} the solution of the invariance problem on simple sets is trivial,
we focus on the case when $\kappa$ is absorbing and select the parameters $\alpha:= (1-\mu)\kappa$, $\beta := -\sigma\kappa$, where
$\kappa \in \R$ is an arbitrary state. The update equation \eqref{eq:example-dynamical} takes the new form:
\begin{equation*}
    \x_{n+1} - \kappa = \mu(\x_k-\kappa) + \sigma(\x_k - \kappa)\xi_k,
\end{equation*}
and by applying a shift on $\kappa$,
without loss of generality we can focus on the following model:
\begin{equation}\label{eq:example-dynamical.upd}
   \x_{n+1} = \mu \x_n + \sigma \x_n\cdot \xi_n.
\end{equation}
In the latter equation $\sigma$ can be assumed to be positive,
since $\xi_n$ has a symmetric distribution.
The kernel associated to the \dt-MP \eqref{eq:example-dynamical.upd}
is weakly continuous and takes the following form:
\begin{equation*}
    P(x,A) =
    \begin{cases}
       \frac{1}{\sigma|x|\sqrt{2 P i}}\int_A \mathrm e^{-\frac{(t-\mu x)^2}{2(\sigma x)^2}}\,\d t&,\text{ if }x\neq 0,
          \\
       1_{A}(0)&,\text{ if }x = 0.
    \end{cases}
\end{equation*}
Since the compact set $A$ is non-simple if and only if $0\in A$,
let us consider the invariance problem for the set $A = [-1,1]$. The discussion above suggests that $\las(A) = \{0\}$,
so $u(0;A) = 1$.
For $x\neq 0$,
let us relate the original process $X$ to the random walk (see e.g. \cite[Chapter 4]{D04}.)
Define $Y_n:= \log|\x_n|$, so the update equation becomes:
\begin{equation*}
  Y_{n+1} = Y_n+\log |\mu+\sigma \xi_n|,
\end{equation*}
where $Y_0 = y := \log{|x|}$.
The expected value of the increment of the random walk
\begin{equation*}
    h(\mu,\sigma) := \mathsf E\log |\mu+\sigma \xi_1|
\end{equation*}
determines its asymptotic behavior.
In particular, $\limsup_{n\to\infty}Y_n = +\infty$ holds $ \pr_x$-a.s. in case $h(\mu,\sigma)\geq 0$ for all $x\neq 0$,
and $\lim_{n\to\infty}Y_n = -\infty$ holds $ \pr_x$-a.s. in case $h(\mu,\sigma)<0$ for all $x\neq 0$ \cite[Chapter 4]{D04}.
As a result, if the values of the parameters $\mu,\sigma$ are such that $h(\mu,\sigma)\geq 0$,
we obtain that, for any $x\neq 0$, the following holds:
\begin{equation*}
  u(x;A) =  \pr_x\left\{\sup\limits_{n\geq 0}|\x_n|\leq 1\right\} =  \pr_y\left\{\sup\limits_{n\geq 0}\log |\x_n|\leq 0\right\} = 0,
\end{equation*}
which allows to conclude that in this case $u(x;A) = 1_{\{0\}}(x)$.

We are left with the case $h(\mu,\sigma)<0$.
Since $P(\cdot,\{0\})$ is not a continuous function on $A$,
Assumption \ref{as:1} does not hold and thus Theorem \ref{thm:Bellman-w.mod} cannot be applied.
We then resort to Theorem \ref{thm:decomp},
which requires finding a $\delta$-locally excessive function.

Let us fix $\mu,\sigma$ and consider $g_q(x):=|x|^q$, for $q\geq 0$.
If we define $b(q) = \mathsf E|\mu+\sigma \xi_1|^q$,
then clearly $ P  g_q(x) = b(q)\cdot g_q(x)$, so $g_q$ is $\delta$-locally excessive if and only if $b(q)<1$.
We obtain $b(0) = 1$ and $b'(0) = h(\mu,\sigma)$.
Recall that we are now interested in the case $h(\mu,\sigma)<0$,
which leads to conclude that there always exists a $q>0$ such that the function $g_q(x) = |x|^q$ is $\delta$-locally excessive. Hence, for $h(\mu,\sigma)<0$ and such $q$, Theorem \ref{thm:decomp} can be applied to find the solution of the invariance problem using $g_q$ as a $1$-locally excessive function.
More precisely, according to \eqref{eq:decomp-Lyap} adapted to the special case of the invariance problem, we obtain:
\begin{equation*}
  0\leq w(x;A,(-\sqrt[q]{\ve},\sqrt[q]{\ve}))-u(x;A)\leq \ve,
\end{equation*}
and function $w(x;A,(-\sqrt[q]{\ve},\sqrt[q]{\ve}))$ can be computed, since $m(A\setminus (-\sqrt[q]{\ve},\sqrt[q]{\ve}))<\infty$ as it follows from Theorem \ref{thm:decomp}.

Finally, let us add a comment on the lack existence of a $\delta$-locally excessive function for a weakly continuous \dt-MP.
Consider the case in \eqref{eq:example-dynamical.upd} with parameters $h(\mu,\sigma)\geq 0$, so that $u(x;A) = 1_{\{0\}}(x)$.
If there existed a function $g$ that is $\delta$-locally excessive on $A = [-1,1]$ then
$
    u(x;A)\geq u(x;\{g<\delta\})\geq 1-\frac{g(x)}{\delta},
$
which implies that $u(x;A)>0$ in some neighborhood of $\{0\}$: this leads to a contradiction.

\subsection{A two-dimensional non-linear Gaussian system}\label{ssec:cs.2}

Let us provide a more computational example for the application of the methods developed in this work.
Let $X = \R^2$ be endowed with the standard topology,
and consider a \dt-MP with dynamics given by the following system of non-linear difference equations:
\begin{equation}\label{eq:example-dynamical.2}
  \begin{cases}
    \x_{1,n+1} &= 0.5 \x_{2,n}(3\x_{1,n}^2  + 2X^2_{2,n}-0.5) + 0.6 \eta_n\sqrt{X^2_{1,n} + X^2_{2,n}},
    \\
    \x_{2,n+1} &= 0.9 \x_{2,n}(2\x_{1,n}^2 + 4\x_{1,k}\x_{2,n} +3X^2_{2,n}-0.5) + 0.6 \zeta_n\sqrt{X^2_{1,n} + X^2_{2,n}},
  \end{cases}
\end{equation}
where $(\x_{1,0},\x_{2,0}) = (x_1,x_2) = x$.
Here $(\eta_n)_{n\geq 0}$ and $(\zeta_k)_{k\geq 0}$ are independent sequences of iid standard normal random variables.

The process $\x$ is weakly continuous and its kernel can be expressed explicitly as in Section \ref{ssec:cs.1}.
Notice that the origin $\{0\}$ is the only bounded absorbing set.
We are interested in the solution of the infinite-horizon invariance problem over the compact set $A = [-0.6,0.6]\times [-0.6,0.6]$.
Again, $P(\cdot,\{0\})$ is not a continuous function on $A$ so that Assumption \ref{as:1} is not satisfied,
and thus Theorem \ref{thm:Bellman-w.mod} cannot be applied as discussed in Section \ref{ssec:cs.1}.
It is thus necessary to find a $\delta$-locally excessive function on $A$.
Let us start by discussing the behavior of the process $\x$ on the phase plane.
For $x$ far from the origin, the non-linear terms (appearing in brackets in \eqref{eq:example-dynamical.2}) play a more important role than the linear ones,
whereas for $x$ close to the origin the situation is reversed.
We then expect that a function measuring the distance from the origin may be locally excessive.
For this reason, we consider $g(x) = \|x\|^2$, which leads to the following:
\begin{equation*}
    \begin{split}
         P  g(x_1,x_2) &= \frac{1}{200}(144x_1^2+197x_2^2-474x_1^2x_2^2+1098x_1^4x_2^2-648x_1x_2^3)\\&+\frac{1}{200}(2592x_1^3x_2^3-586x_2^4+5136x_1^2x_2^4+3888x_1x_2^5+1658x_2^6).
    \end{split}
\end{equation*}
It holds that $\{g<0.25\}\subseteq \e_g$,
hence $g$ is $\delta$-locally excessive on $A$, with $\delta = 0.25$.
Figure \ref{fig:cs1} shows sets $\{g<0.25\}, \e_g$, and $A$.
Set $A$ intersects $\e^c_g$,
which can be interpreted as follows:
starting from a state $x\in A\cap \e_g$,
process $\x$ exhibits contractive dynamics,
whereas for $x\in A\cap \e_g^c$ the difference $ P  g(x)-g(x)$ is positive and gets larger as $\|x\|$ grows,
hence trajectories initialized in $x\in A\cap \e_g^c$ expand away from the origin.
Based on this consideration we expect clear differences between the values of function $u(x;A)$ for $x\in A\cap\e_g$ and $x\in A\cap\e_g^c$.
\begin{figure*}[ht]
    \centering
        {\psfrag{aa}{$x_1$}\psfrag{bb}{$x_2$}
        \includegraphics[keepaspectratio=true,width=11cm]{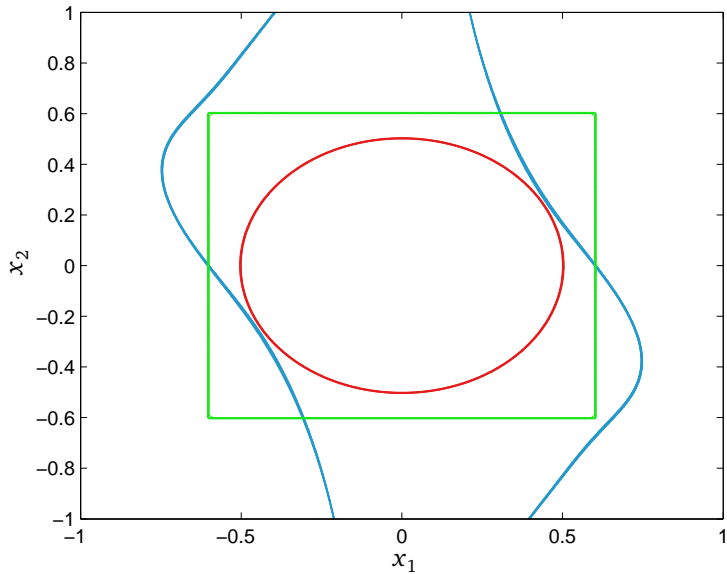}
        }
    \caption{Infinite-horizon invariance problem over set $A$.
    The boundaries of sets $\e_g$ (dark blue curves), $A$ (green square) and $\{g<0.25\}$ (red circle).}
\label{fig:cs1}
\end{figure*}

We apply the decomposition technique in Theorem \ref{thm:decomp},
where by selecting an $\varepsilon = 0.02$ we obtain that $0\leq w(x;A,\{g<5\cdot 10^{-3}\})-u(x;A)\leq 0.02$.
To simplify the calculations,
we consider the function $w(x;A,B)$ for
\begin{equation*}
  B = (-0.05.0.05)\times(-0.05,0.05)\subset \{g<5\cdot 10^{-3}\}
\end{equation*}
and as a result we have $0\leq w(x;A,B)-u(x;A) \leq 0.02$.
To compute the values of function $w$ we use the bounds provided in Proposition \ref{prop:cmp}: in this case $m(A\setminus B) = 1$ and $\rho(A\setminus B) \approx 0.957$, so by considering $n=50$ iterations we obtain
\begin{equation*}
  0\leq w(x;A,B)-w_n(x;A,B) \leq 0.112.
\end{equation*}

Thus far, the methods developed in this paper (in particular Theorem \ref{thm:decomp}),
have allowed us to reduce the infinite-horizon invariance problem over a non-simple set to a finite-horizon reach-avoid problem.
Let us now mention how the value function corresponding to the latter problem can be computed.
The calculation of the value function $w_n$ is performed with a target error $0.1$,
which is achieved by employing a standard uniform discretization algorithm \cite{aklp2010} -- thus the resulting overall error equals to $0.232$.
Based on the time horizon of the problem, and due to the degenerate nature of the kernel in the neighborhood of the origin, and the fine size selected by the partitioning procedure to achieve the small required precision, the computation took $24$ hours on Intel Core i5, 2.4 GHz with 4Gb RAM.
This computational time can be further reduced by leveraging more involved numerical procedures \cite{sa2011},
which however are outside of the scope of this study.

\begin{figure*}[ht]
    \centering
        \subfigure[Local excessivity of $g$ on the set $A$]{
        {\psfrag{d}{$x_1$}\psfrag{e}{$x_2$}\psfrag{f}{$Pg-g$}
        \includegraphics[keepaspectratio=true,width=1\textwidth]{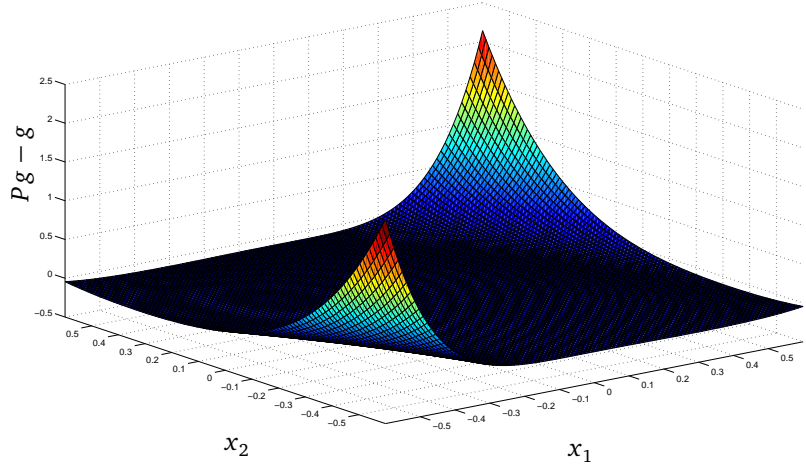}
        }
        }
        \subfigure[Invariance value function]{
        {\psfrag{a}{$x_1$}\psfrag{b}{$x_2$}\psfrag{c}{$u$}
        \includegraphics[keepaspectratio=true,width=1\textwidth]{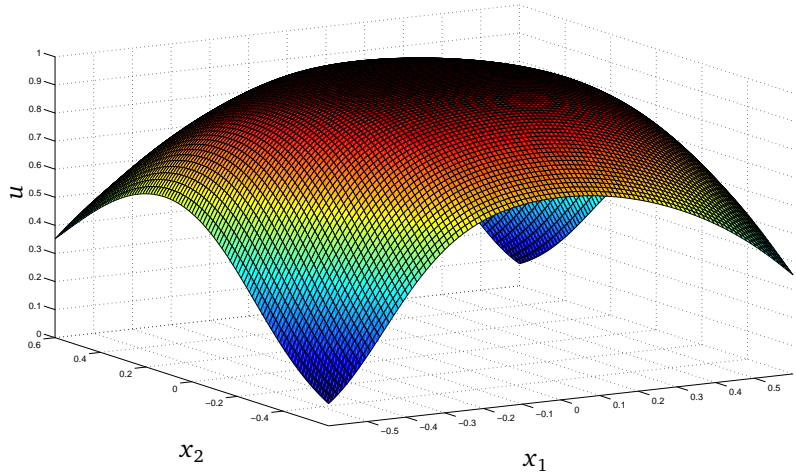}}
        }
    \caption{Results for the infinite-horizon invariance problem on the set $A$.
    Graphs of functions $ P  g-g$ (a) and $u$ (b).}\label{fig:cs}
\end{figure*}

The first goal of this case study was to show that infinite-horizon problems can be solved efficiently, with strict bounds on the error, even in the case of nonlinear dynamics and kernels which admit non-trivial absorbing sets.
The use of the decomposition technique has also allowed us to avoid computations over the neighborhood of the absorbing set $(0,0)$ where the kernel $P$ degenerates.
In particular, it is important for numerical methods based on the discretization of the state space,
since their error bounds depends on Lipschitz constants of densities.
Moreover, with this approach the error of computation can be made as small as needed by varying the error $\ve$ related to the decomposition,
the number $n$ of iterations for the reach-avoid problem,
and the grid size for the discretization.

As already mentioned, the choice of the set $A$ with regards to the excessive region plays an important role. On Figure \ref{fig:cs}(a) one can observe large positive values of $ P  g(x)-g(x)$ for $x$ close to points $(-0.6,0.6)$ or $(0.6,0.6)$. We expect a diverging behavior of $\x$ when starting in that region. This fact is clearly shown on Figure \ref{fig:cs}(b) where the invariance value function $u$ takes the smallest values exactly in that region.

\section{Conclusions}\label{sec:concl}

This work has provided a general framework for the study of formal algorithms for PCTL verification of discrete-time Markov processes over general state spaces.
The main focus of the article
has been placed on the verification of infinite-horizon PCTL specifications,
both in terms of characterization of the given PCTL formula and in terms of precise numerical computation of the corresponding value function.
It has been shown that structural properties of the stochastic kernel,
namely the possible presence of absorbing subsets of given sets,
are crucial for problems over the infinite horizon.
In particular, the solution of the invariance is either trivial (on simple sets) or extremely complicated (on non-simple sets).
This has lead to criteria to distinguish such instances and to techniques to tackle the latter case -- these techniques have been illustrated by two case studies.

The outcome of this work is that infinite-horizon problems cannot in general be solved exactly or algorithmically.
However, precise reduction of these problems to finite-horizon analogues allows tapping on techniques for the latter, thus inheriting their scalability.
This leads to an emphasis on the verification of the simplicity of a given set and on the development of procedures to find $\delta$-locally excessive functions.

These questions set compelling goals to the authors and are to be further pursued in future work,
along with the application of the developed methods to other classes of specifications (beyond PCTL).
Furthermore, extensions to continuous-time and control-dependent models are also deemed research worthy.


\bibliographystyle{amsalpha}
\bibliography{../../../my_bib}

\section{Appendix}\label{sec:appendix}

Theorem \ref{thm:main} requires the compactness of the set $A$ and the weak continuity of the kernel $P$, however some of the relations between statements in this theorem are true in the general case as it is shown in Figure \ref{fig:relax}.
\begin{figure}[h!]
    \begin{center}
    \label{fig:relax}
        \begin{tikzpicture}[>=triangle 60,every node/.style={draw,circle,minimum width={2em},node distance=4em}]
            \node (a) {5};
            \node [below left of=a] (b) {2};
            \node [left of=b] (d) {1};
            \node [below right of=a] (c) {3};
            \node [right of=c] (e) {4};
            \draw [->] (b) -- (a);
            \draw [->] (c) -- (a);
            \draw [<-] (c) -- (b);
            \draw [<->] (b) -- (d);
            \draw [<->] (c) -- (e);
        \end{tikzpicture}
        \caption{Generalization of the relations between statements of Theorem \ref{thm:main}.}
    \end{center}
\end{figure}
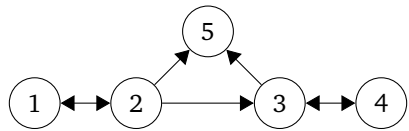
First of all, for the pair \textit{1)} $\Leftrightarrow$ \textit{2)} it is clear that \textit{2)} is a stronger statement in general. Moreover, from the proof of Theorem \ref{thm:main} it clearly follows that \textit{1)} $\Rightarrow$ \textit{2)} without any assumptions on $A$ and $P$. For \textit{3)} $\Leftrightarrow$ \textit{4)} the following holds:

\begin{proposition}\footnote{This proposition generalizes a result from \cite[Proposition 9]{rcsl2010},
where the trivial invariance was shown to be sufficient for the uniqueness over a smaller class of functions.}\label{prop:uniqueness}
  Equation \eqref{eq:Bellman-u} admits a unique solution if and only if $u(\cdot;A)\equiv 0$.
\end{proposition}

\begin{proof}
  Equation \eqref{eq:Bellman-u} is linear, so if its solution is unique it is the trivial zero solution. Since $u(\cdot;A)$ is one of solutions, $u(\cdot;A)\equiv 0$.

  Conversely, let us suppose that $u(\cdot;A)\equiv 0$ and let $f\in \B$ be any other solution of \eqref{eq:Bellman-u}, so that $\|f\|>0$. Clearly, the function $\tilde f:= \frac{f}{\|f\|}$ is also a solution of this equation and $\tilde f\leq 1$. As a result, it holds that $\tilde f\leq u = 0$ (see Remark \ref{rem:least-fixpoint}) so that $\tilde f\leq 0$. On the other hand, $-\tilde f$ is also a solution of \eqref{eq:Bellman-u} due to the linearity of the equation and $-\tilde f\leq u\equiv 0$ which leads to $\tilde f = 0$. However, we have $\|\tilde f\| = 1$ by definition, hence we come to a contradiction.
\end{proof}

Now we only left to discuss relations between \textit{2)}, \textit{3)} and \textit{5)}. From contraction mapping theorem it follows that \textit{2)} $\Rightarrow$ \textit{3)}. Moreover, if $u(\cdot;A)\equiv 0$ then $A$ is simple since $\las(A) = \{u(\cdot;A) = 1\}$ is empty in this case,
so that \textit{3)} $\Rightarrow$ \textit{5)}.
As a result, all the relations in Figure \ref{fig:relax} are true.
On the other hand,
examples below show that if either $A$ is not compact,
or $P$ is not weakly continuous,
then converse relations between \textit{2)}, \textit{3)} and \textit{5)} do not hold true.

\smallskip

We first show that the weak continuity is not sufficient without the compactness.

\textbf{1.} Let us show that  \textit{3)}$+$\textit{5)} $\nRightarrow$ \textit{2)}. Consider an example from Section \ref{ssec:cs.1} given by the equation \eqref{eq:example-dynamical.upd} with $\mu = 0$ and $h(0,\sigma)\geq 0$. Let us choose the set $A = [-1,1]$ so as it has been proved, $u(x;A) = 1_{\{0\}}(x)$. Let us put $\tilde A = A\setminus \{0\}$, i.e. it is not compact. By induction it can be proved that $u_n(x;A) - u_n(x;\tilde A) = 1_{\{0\}}(x)$ for all $n\geq 0$ since it holds for $n = 0$ and
\begin{equation*}
    \begin{split}
    u_{n+1}(x;A) - u_{n+1}(x;{\tilde A}) &= 1_A(x)\int_X u_{n}(y;A)P(x,\d y)-1_{\tilde A}(x)\int_X u_{n}(y;{\tilde A})P(x,\d y)
    \\
        &=1_{\{0\}}(x) +1_{\tilde A}(x)\int_{\{0\}}P(x,\d y) = 1_{\{0\}}(x).
    \end{split}
\end{equation*}
Note, however that if $f\in \B$ is continuous on $A = [-1,1]$ so is $\i_A f$ due to the structure of the kernel $P$. As a result, functions $u_n(\cdot;A)$ are continuous on $A$ and since $u_n(0;A) = 1$ for all $n\geq 0$ it holds that $\|u_n(\cdot;\tilde A)\| = \|u_n(\cdot;A) - 1_{\{0\}}(\cdot)\| = 1$. Although the set $\tilde A$ is simple, its invariance value function $u(\cdot;\tilde A)\equiv 0$ and the uniqueness for the solution of \eqref{eq:Bellman-u} holds, we have $m(A) = \infty$ which proves that \textit{3)}$+$\textit{5)} $\nRightarrow$ \textit{2)} in general.

\textbf{2.} Let us show that \textit{5)} $\nRightarrow$ \textit{3)}. Following the same lines as above, we consider \eqref{eq:example-dynamical.upd} with $\mu = 0$ and $h(0,\sigma)<0$. As it has been discussed in Section \ref{ssec:cs.1} the function $u(x;A)$ for $A = [-1,1]$ is positive in the neighborhood of $0$. Still, it holds that $u(x;\tilde A) = u(x;A) - 1_{\{0\}}(x)$ for $\tilde A = A\setminus \{0\}$ so the invariance value function for the simple non-compact set $\tilde A$ is non-trivial and hence the solution of \eqref{eq:Bellman-u} is not unique.

\smallskip

Let us now show that if $A$ is compact but the weak continuity assumption on $P$ is relaxed then \textit{3)}$+$\textit{5)} $\nRightarrow$ \textit{2)} and \textit{5)} $\nRightarrow$ \textit{3)}. To do this, one should make similar considerations as in \textbf{1.} an \textbf{2.} for the same kernels, just redefining $P(0,\{2\}) = 1$.
\end{document}